\documentclass[twoside,bezier,12pt, reqno]{amsart}
\usepackage{amsmath, amsthm, amscd, amsfonts,url, amssymb,graphicx,graphics, color}
\usepackage[bookmarksnumbered, plainpages]{hyperref}
\usepackage{pgf,tikz}
\usetikzlibrary{arrows}
\newtheorem{theorem}{Theorem}[section]

\newtheorem{proposition}[theorem]{Proposition}

\theoremstyle{definition}
\newtheorem{definition}[theorem]{Definition}

\theoremstyle{remark}

\numberwithin{equation}{section}
\begin{document}

\title{A   class of  integral  graphs constructed from the hypercube}
\author{ S.  Morteza Mirafzal }%

\address{ Department of Mathematics, Lorestan University, Khoramabad, Iran}%
\email{smortezamirafzal@yahoo.com}
\email{mirafzal.m@lu.ac.ir}%

\begin{abstract}In this paper,  we determine the set of all distinct eigenvalues  of the line graph which is  induced by the  first and second layers of the  hypercube $ Q_n $, $n>3$. We show that this graph has precisely five distinct eigenvalues and  all of its eigenvalues are integers.  \\


\


\end{abstract}

\maketitle
% ----------------------------------------------------------------

%**************************************************************************************************************
\section{ Introduction}
\noindent
 In this paper, a graph $\Gamma=(V,E)$ is
considered as an undirected simple graph where $V=V(\Gamma)$ is the vertex-set
and $E=E(\Gamma)$ is the edge-set. For all the terminology and notation
not defined here, we follow [1,4,5].

Let $ n \geq 1 $ be an integer.  The hypercube  $Q_n$ is the graph whose vertex set is $ \{0,1  \}^n $, where two $n$-tuples  are adjacent if  they differ in precisely one coordinates. In the graph $Q_n$, the layer $L_i$ is the set of vertices which contain $i$ 1s, namely,  vertices of weight $i$, $ 1 \leq i \leq n$.  We denote by $ {Q_n}(i,i+1)$, the subgraph of $Q_n$ induced by layers $L_i$ and $ L_{i+1} $.
In this paper, we want to determine the set of all distinct eigenvalues of line graph of the graph $ {Q_n}(1,2) $. \

 We can consider the graph $Q_n$ from another point of view. The Boolean lattice $BL_n, n \geq 1$, is the graph whose vertex set is the set of all subsets of $[n]= \{ 1,2,...,n \}$, where two subsets $x$ and $y$ are adjacent if their symmetric difference has precisely one element.  In the graph $BL_n$, the layer $L_i$ is the set of $i-$subsets of $[n]$.  We denote by $ {BL_n}(i,i+1)$, the subgraph of $BL_n$ induced by layers $L_i$ and $ L_{i+1} $. It is an easy task to show that the graph $Q_n$ is isomorphic with the graph $BL_n$, by an isomorphism that induces an isomorphism from   $ {Q_n}(i,i+1)$ to  ${BL_n}(i,i+1)$.    \

In the sequel, we denote the graph ${BL_n}(1,2)$ by $H(n)$.  Therefore,    the graph $ H(n)$ is a
 graph with vertex set \

\

$V=\{v | v \subset [n],  | v |  \in \{ 1,2  \} \} $ and the
edge set \

\

$ E= \{ \{ v , w \} | v , w \in V , v \subset w $ or $ w \subset v \} $. \

\

It is clear that $ H(n) $ is a bipartite graph
with cells $V_1,V_2$ where\

 $V_1=  \{ v| v\subset [n], |v|=1 \}$ and $ V_2 = \{   v| v\subset [n], |v|=2 \}$. \

 Also, if $v\in V_1$, then $ deg(v)=n-1 $ whereas if $v \in V_2$, then
$deg(v) =2$, hence $H(n)$ is not a regular graph. Now, it is obvious that $H(n)$ has $n(n-1)$ edges.
The  following figure shows  $ H(5)$    in the plane.\

\

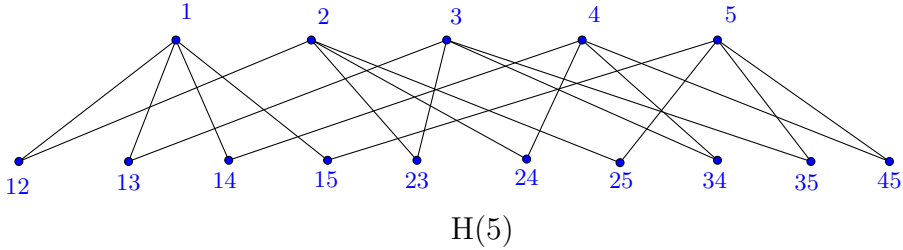
\definecolor{qqqqff}{rgb}{0.,0.,1.}
\begin{tikzpicture}[line cap=round,line join=round,>=triangle 45,x=.9cm,y=.8cm]
\clip(-4.56,-0.04) rectangle (8.9,4.78);
\draw (-2.,4.)-- (-4.32,1.98);
\draw (-2.,4.)-- (-2.7,1.98);
\draw (-2.,4.)-- (-1.22,2.);
\draw (-2.,4.)-- (0.24,2.);
\draw (0.,4.)-- (-4.32,1.98);
\draw (0.,4.)-- (1.56,2.);
\draw (0.,4.)-- (3.18,2.02);
\draw (0.,4.)-- (4.56,1.96);
\draw (2.,4.)-- (-2.7,1.98);
\draw (2.,4.)-- (1.56,2.);
\draw (2.,4.)-- (6.,2.);
\draw (2.,4.)-- (7.38,1.98);
\draw (4.,4.)-- (-1.22,2.);
\draw (4.,4.)-- (3.18,2.02);
\draw (4.,4.)-- (6.,2.);
\draw (4.,4.)-- (8.54,1.98);
\draw (6.,4.)-- (0.24,2.);
\draw (6.,4.)-- (4.56,1.96);
\draw (6.,4.)-- (7.38,1.98);
\draw (6.,4.)-- (8.54,1.98);
\draw (1.9,1.26) node[anchor=north west] {H(5)};
\begin{scriptsize}
\draw [fill=qqqqff] (-2.,4.) circle (1.5pt);
\draw[color=qqqqff] (-1.84,4.48) node {$1$};
\draw [fill=qqqqff] (0.,4.) circle (1.5pt);
\draw[color=qqqqff] (0.18,4.4) node {$2$};
\draw [fill=qqqqff] (2.,4.) circle (1.5pt);
\draw[color=qqqqff] (2.14,4.4) node {$3$};
\draw [fill=qqqqff] (4.,4.) circle (1.5pt);
\draw[color=qqqqff] (4.18,4.42) node {$4$};
\draw [fill=qqqqff] (6.,4.) circle (1.5pt);
\draw[color=qqqqff] (6.2,4.42) node {$5$};
\draw [fill=qqqqff] (-4.32,1.98) circle (1.5pt);
\draw[color=qqqqff] (-4.34,1.56) node {$12$};
\draw [fill=qqqqff] (-2.7,1.98) circle (1.5pt);
\draw[color=qqqqff] (-2.7,1.64) node {$13$};
\draw [fill=qqqqff] (-1.22,2.) circle (1.5pt);
\draw[color=qqqqff] (-1.28,1.66) node {$14$};
\draw [fill=qqqqff] (0.24,2.) circle (1.5pt);
\draw[color=qqqqff] (0.22,1.68) node {$15$};
\draw [fill=qqqqff] (1.56,2.) circle (1.5pt);
\draw[color=qqqqff] (1.56,1.66) node {$23$};
\draw [fill=qqqqff] (3.18,2.02) circle (1.5pt);
\draw[color=qqqqff] (3.18,1.72) node {$24$};
\draw [fill=qqqqff] (4.56,1.96) circle (1.5pt);
\draw[color=qqqqff] (4.58,1.66) node {$25$};
\draw [fill=qqqqff] (6.,2.) circle (1.5pt);
\draw[color=qqqqff] (5.96,1.66) node {$34$};
\draw [fill=qqqqff] (7.38,1.98) circle (1.5pt);
\draw[color=qqqqff] (7.32,1.64) node {$35$};
\draw [fill=qqqqff] (8.54,1.98) circle (1.5pt);
\draw[color=qqqqff] (8.54,1.66) node {$45$};
\end{scriptsize}
\end{tikzpicture} \

     We can see, by an easy argument
that the graph $H(n)$ is
connected and its diameter is 4. \

 We now consider the line graph of $H(n)$. We denote by $L(n)$ the line graph of the graph $H(n)$. Then each vertex of $L(n)$ is of the form
$ \{v,w\} = \{ \{ i \},  \{ i,j \}\}$, where $i,j \in [n]$, $i \neq j$, and two vertices $ \{v,w\}, \{ u,s  \}$ are adjacent  whenever the intersection of them is a set of order 1. In the sequel,  we denote the vertex  $\{ \{ i \},  \{ i,j \} \}$ by $[i, ij]$. Thus, if $[i,ij]  $ is a vertex of $ L(n) $, then \

\

$  N([i,ij])  = \{ [i,ik] | j,i \neq k \in [n] \} \cup \{ [j,ij] \}$. \

\

Hence,  $  deg ([i,ij]) = n-1$, in other words,  $L(n)$ is a regular graph of valency $n-1$.   By an easy argument, we can show that the graph $L(n)$ is a connected graph with diameter 3. Also, its girth is 3 and hence it is not a bipartite graph.  The following figure displays $L(4)$ in the plane.
 Note that in the following figure the vertex $ [i,ij] $ is denoted by,  $i,ij$. \

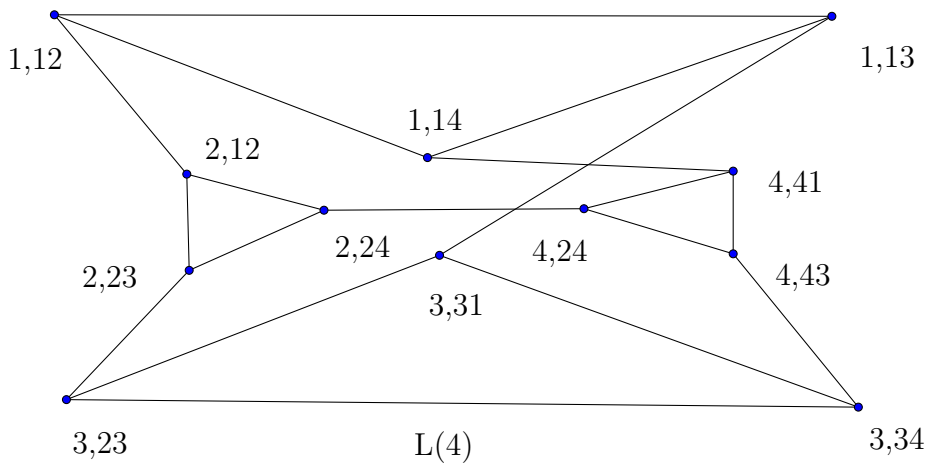
\definecolor{qqqqff}{rgb}{0.,0.,1.}
\begin{tikzpicture}[line cap=round,line join=round,>=triangle 45,x=1.6cm,y=1.0cm]
\clip(-0.7,-1.) rectangle (7.12,5.5);
\draw (-0.02,0.08)-- (3.08,2.);
\draw (6.56,-0.02)-- (3.08,2.);
\draw (-0.02,0.08)-- (6.56,-0.02);
\draw (-0.12,5.2)-- (2.98,3.3);
\draw (2.98,3.3)-- (6.34,5.18);
\draw (6.34,5.18)-- (-0.12,5.2);
\draw (0.98,3.08)-- (2.12,2.6);
\draw (2.12,2.6)-- (1.,1.8);
\draw (0.98,3.08)-- (1.,1.8);
\draw (4.28,2.62)-- (5.52,3.12);
\draw (4.28,2.62)-- (5.52,2.02);
\draw (5.52,2.02)-- (5.52,3.12);
\draw (-0.6,4.9) node[anchor=north west] {1,12};
\draw (6.48,4.92) node[anchor=north west] {1,13};
\draw (2.72,4.08) node[anchor=north west] {1,14};
\draw (1.04,3.7) node[anchor=north west] {2,12};
\draw (0.02,2.) node[anchor=north west] {2,23};
\draw (2.12,2.4) node[anchor=north west] {2,24};
\draw (3.76,2.34) node[anchor=north west] {4,24};
\draw (5.72,3.26) node[anchor=north west] {4,41};
\draw (5.78,2.02) node[anchor=north west] {4,43};
\draw (6.56,-0.16) node[anchor=north west] {3,34};
\draw (-0.06,-0.22) node[anchor=north west] {3,23};
\draw (2.9,1.62) node[anchor=north west] {3,31};
\draw (2.12,2.6)-- (4.28,2.62);
\draw (5.52,3.12)-- (2.98,3.3);
\draw (-0.12,5.2)-- (0.98,3.08);
\draw (5.52,2.02)-- (6.56,-0.02);
\draw (1.,1.8)-- (-0.02,0.08);
\draw (6.34,5.18)-- (3.08,2.);
\draw (2.78,-0.2) node[anchor=north west] {L(4)};
\begin{scriptsize}
\draw [fill=qqqqff] (2.98,3.3) circle (1.5pt);
\draw [fill=qqqqff] (-0.12,5.2) circle (1.5pt);
\draw [fill=qqqqff] (6.34,5.18) circle (1.5pt);
\draw [fill=qqqqff] (3.08,2.) circle (1.5pt);
\draw [fill=qqqqff] (-0.02,0.08) circle (1.5pt);
\draw [fill=qqqqff] (6.56,-0.02) circle (1.5pt);
\draw [fill=qqqqff] (2.12,2.6) circle (1.5pt);
\draw [fill=qqqqff] (0.98,3.08) circle (1.5pt);
\draw [fill=qqqqff] (1.,1.8) circle (1.5pt);
\draw [fill=qqqqff] (4.28,2.62) circle (1.5pt);
\draw [fill=qqqqff] (5.52,3.12) circle (1.5pt);
\draw [fill=qqqqff] (5.52,2.02) circle (1.5pt);
\end{scriptsize}
\end{tikzpicture}

The graph $L(n)$ has various interesting properties, amongst
of them, we interested in  its spectrum.
In the present paper, we show that $L(n)$ has precisely  5
distinct eigenvalues. Also, we show that each eigenvalue of $L(n)$ is an integer.

\section{ Preliminaries }

Let $ \Gamma  $ be a graph with vertex set $  V= \{ v_1,v_2,...,v_n \} $ and
edge set $ E( \Gamma ) $. The adjacency matrix $A = A(\Gamma) = [a_{ij} ]$ of $ \Gamma  $ is an $n \times n$ symmetric
matrix of $0^,$s and $1^,$s with $ a_{ij} = 1$ if and only if $v_i $ and $v_j $ are adjacent. The
characteristic polynomial of $ \Gamma  $ is the polynomial $P(G) = P(G, x) = det(xI_n - A)$,
where $I_n $ denotes the $n \times n$   identity matrix. The spectrum
of $A(\Gamma)$ is also called the spectrum of $\Gamma  $. If the eigenvalues of $\Gamma$ are ordered by
$ \lambda_1 > \lambda_2 > ... > \lambda_r  $, and their multiplicities are $ m_1, m_2,...,m_r $,  respectively,
 then we write \

\

 \centerline{ $   Spec( \Gamma)$    =  ${ \lambda_1,\lambda_2,..., \lambda_r } \choose{ m_1, m_2,...,m_r }  $
 or $ Spec( \Gamma)$    = $ \{  \lambda_1^{m_1},  \lambda_2^{m_2},..., \lambda_r^{m_r}   \} $} \

If all the eigenvalues of the adjacency matrix
 of the graph $\Gamma $ are integers, then we say that $\Gamma $ is an integral graph.
The notion of integral graphs was first introduced by F. Harary and A.J. Schwenk
in 1974 (see [6]).  In 1976 Bussemaker and Cvetkovic [3], proved that there are exactly 13 connected
cubic integral graphs. In general, the problem of characterizing
integral graphs seems to be very difficult. There are good surveys in this area ( for example [2] ).\

 Let $X$ be a graph with vertex set $V$. We have the following definitions and facts [5].

\begin{definition}
A partition $\pi$ of $V$ with cells $C_1, . . . , C_r$ is equitable if the number
of neighbours in $C_j$
of a vertex $v \in C_i$  is a  constant and  depends only on the choice of $C_i$
and $C_j$.
In
this case, we denote the number of neighbors in $C_j$
of any vertex in $C_i$  by $p_{ij}$.

\end{definition}
It is clear that if $\pi$ is  an equitable partition with cells $C_1, . . . , C_r$
then every
vertex in $C_i$ has the same valency.
\begin{definition}
Let $X$  be a graph. If $H \leq Aut(X)$ is a group of automorphisms of $X$, then $ H $ partition the vertex set of $ X$ into orbits. The partition of $X$ consisting of the set of orbits which are  constructed by $H$,
is called an orbit partition of $X$.

\end{definition}

\begin{definition}
Let $X$ be  a graph with equitable partition $ \pi = \{ C_1,...,C_r \} $. The directed graph with vertex set $\pi $ with $ b_{ij} $ arcs from $C_i$ to $C_j$ is called the quotient of $X$
over $\pi$ and is denoted by $ X / \pi $.
\end{definition}

Therefore, if $ P=(p_{ij}) $ is the adjacency matrix of the directed  graph $ X / \pi $, then $p_{ij}$ is the number
of neighbours in $C_j$ of any vertex in $C_i$
\begin{theorem}
Let $X$ be a graph with equitable partition $\pi$. Let $P$
be the adjacency matrix of the directed graph $ X/\pi $
and $A$ be the adjacency  matrix of $X$. Then each eigenvalue of the matrix $P$ is an eigenvalue of the matrix $A$.
\end{theorem}

\begin{theorem}
Let $X$  be a vertex-transitive  graph and $H \leq Aut(X)$ is a group of automorphisms of $X$. If $\pi$, the orbit partition of $H$,  has a singleton cell $ \{ u \}$, then every eigenvalue of $X$ is an eigenvalue of $X/ \pi$.
\end{theorem}

\begin{proposition}
Every orbit partition is an equitable partition.

\end{proposition}

\begin{proof}
Let $X$ be a graph and $H$ be a subgroup of the automorphism group of $X$. Let $v,w \in V=V(X)$ and $O_1= H(v), \ O_2= H(w)$  are    orbits  of $H$. Let $v$ has $r$ neighbours  in $O_2$, and $u$ has $s$ neighbours in $O_2$.   Let $ N(v) \cap O_2$
=$  \{ w_1,...,w_r  \}  $. Since $u$ is a vertex in $O_1$, then there is some $h \in H$ such that $u=h(v)$, and thus we have; \

\centerline{ $ N(u) \cap O_2 =N( h(v) ) \cap O_2 = h( N(v) ) \cap O_2=h( \{w_1,...,w_r  \}) \cap O_2 $ = }
 \centerline{$ \{ h(w_1),...,h(w_r) \} \cap O_2  = \{ h(w_1),...,h(w_r) \} $} \
Note that since $w_i \in H(w) =O_2, 1 \leq i \leq r$, then $ h(w_i) \in hH(w) = H(w)=O_2$, because $ H$ is a group. Now,
Our argument shows that $ r =s $.

\end{proof}

\section{Main results}\
\
\begin{proposition}
The  graph $L(n)$ is a vertex-transitive graph.

\end{proposition}

\begin{proof}
We let $ \Gamma = L(n)$.  Note that the  group $ Sym([n]) $ is a subgroup of the group $ Aut(\Gamma ) $. If $ \theta \in  Sym([n]) $, then $ f_{ \theta } $ defined by the rule $ f_{ \theta }([i,ij]) =[\theta (i),  \theta (i) \theta (j)  ] $ is an automorphism of the graph $ \Gamma $. In fact, if $G= \{ f_{ \theta } \ | \  \theta \in Sym(n) \}$, then $G$  is isomorphic with $ Sym([n])  $ and  $ G \subseteq Aut(\Gamma)$.
 Let $ [i,ij] $ and $ [l,lk] $ are arbitrary vertices of $L(n)$. Then for $ \theta = (i,l)(j,k) \in Sym([n]) $, we have $  f_{ \theta }([i,ij]) =[\theta (i),  \theta (i) \theta (j) =[l,lk]     $
\end{proof}

Let $\Gamma = L(n)$. We can determine some of the eigenvalues of $\Gamma$. For example since $\Gamma$ is $(n-1)$-regular, then $n-1$ is the largest eigenvalue of $\Gamma$ [1, chap 3]. On the other hand, since $ \Gamma $ is a line graph, then for each eigenvalue $e$ of $\Gamma$ we have $-2 \leq e$ [5, chap 8]. We show that -2 is an eigenvalue of $ \Gamma$.

\begin{proposition}
Let $\Gamma = L(n)$. Then $ -2 $ is an eigenvalue of $\Gamma$.
\end{proposition}

\begin{proof}
Let $S$ be the subgraph of $ \Gamma$ induced by the vertex set; \

\

\centerline
{ $B= \{ [1,12], [1,13], [3,13], [3,32], [2,32], [2, 21] \}$}
\

It is clear that $S$ is a 6-cycle, and hence -2 is an eigenvalue of the subgraph $S$ [1, chap 3]. Let $ \theta_{min} $ be the minimum of  the eigenvalues of $\Gamma $. Therefore, by  [5, chap 8] we have $  \theta_{min} \leq -2 $. Now by what is stated above, we conclude that  $  \theta_{min} =-2 $.
\end{proof}

We now try to find  some other eigenvalues of the graph $L(n)$, by constructing  a suitable partition of the vertex set of this graph.

\begin{proposition}
Let $\Gamma = L(n)$. Then $ -1 $  and $ n-2 $ are eigenvalues of the graph  $\Gamma$.

\end{proposition}

\begin{proof}
Let $ H = \{  f_{ \alpha } \ | \  \alpha \in Sym([n]), \alpha(1) = 1 \} $.  Then $H$ is a subgroup of $ Aut( \Gamma)$, the automorphism group of $\Gamma $.  Let $ H_1 = \{ \alpha \ | \ f_{ \alpha \  \in H} \}   $.  Note that $H_1$ is a subgroup of $Sym([n])$ isomorphic with $Sym([n-1])$.  In the sequel, we try to find the orbit partition of  $H$.  The group $H$ has three orbits
in its action on the vertex set of $\Gamma$. In fact, we have the following orbits; \newline
$O_1 =H([1,12]) = \{   h([1,12]) \
| \  h \in \
 H \}$ = $ \{ [ \alpha(1), \alpha(1)\alpha(2) ] \
\ | \ \alpha \in H_1\} $=
$ \{  [1,1i] \ | \  i\in \  \{  2,3,...,n \}   \}$.
\newline
$O_2 =H([2,12]) = \{   h([2,12])\
| \  h \in \
 H \}$ = $ \{ [ \alpha(2), \alpha(1)\alpha(2) ] \
\ | \ \alpha \in H_1\} $=
$ \{  [i,1i] \ | \  i\in \  \{  2,3,...,n \}   \}$.
\newline
\

$O_3 =H([3,34]) = \{   h([3,34]) \
| \  h \in \
 H \}$ = $ \{ [ \alpha(3), \alpha(3)\alpha(4) ] \
\ | \ \alpha \in H_1\} $=
$ \{  [i,ij] \ | \  i,j\in \  \{  2,3,...,n \}, i \neq j   \}$. \newline
\

Note that since $ O_1 \cup O_2 \cup O_3 = V= V(\Gamma) $, the group $H$ has  not another orbit. If we let $ \pi = \{ O_1 , O_2, O_3 \} $, then for the matrix $   P = A(\Gamma / \pi)$=$( p_{i,j}) $, the adjacency matrix of the directed graph $\Gamma / \pi $, we have, \newline
\

$p_{11}=n-2$, since $[1,12] \in O_1$ is adjacent to $ n-2 $ other vertices in $O_1$.
\newline
$p_{12}=1$, since $[1,12] \in O_1$ is  adjacent  to 1 vertex in $ O_2 $, namely,  $ [2,12] $.
\newline
$p_{13}=0$, since $[1,12] \in O_1$ is  not adjacent  to any vertex in  $ O_3 $.
\newline
\

$p_{21}=1$, since $[2,12] \in O_2$ is  adjacent  to 1 vertex in  $ O_1 $, namely,  $ [1,12] $.
\newline
$p_{22}=0$, since $[2,12] \in O_2$ is not  adjacent  to any vertex in  $ O_2 $.
\newline
$p_{23}=n-2$, since $[2,12] \in O_2$ is  adjacent  to every vertex in  $ O_3 $ of the form $ [2,i2]$, $i \in \{ 3,...,n \} $.
\newline
\

$p_{31}=0$, since $[3,34] \in O_3$ is not  adjacent  to any vertex in  $ O_1 $.
\newline
$p_{32}=1$, since $[3,34] \in O_3$ is  adjacent  to 1 vertex in  $ O_2 $, namely,  $ [3,31] $.
\newline
$p_{33}=n-2$, since $[3,34] \in O_3$ is  adjacent  to every vertex in $ O_3 $ of the form $ [3,i3]$, $i \in \{ 4,...,n \}$ and the vertex  $[4,34]$.
 Note that since $deg([3,34])$ =$n-1$, this vertex is not adjacent to any other vertex in  $\Gamma$. \newline
\

In other words, we have
 $P$=$\begin{bmatrix}
n-2 & 1 & 0  \\ 1 & 0 & n-2 \\ 0 & 1 & n-2 \\
\end{bmatrix}$. \
Therefore, we have \newline
 $det(xI-P)$ = $(x-(n-2))$ det( $\begin{bmatrix}
x & -(n-2)   \\ -1 & x-(n-2) \\
\end{bmatrix}$) + det (  $\begin{bmatrix}
-1 & -(n-2)   \\  0 & x-(n-2) \\
\end{bmatrix}$)= \

$(x-(n-2)) ( x^2-(n-2)x-(n-2) ) -x + n-2 $= \

$ (x-( n-2 ))( x^2-(n-2)x-(n-2)-1 ) $ = \

 $ (x-( n-2 ))(( x-(n-1) )(x+1)$ \

 We now, can see that the eigenvalues of the matrix $ P$ are $ n-1, n-2$  and  -1. Since by Theorem 2.4.  and  Theorem 2.6. every eigenvalue of $P$ is an eigenvalue of the graph $\Gamma$, thus  $ -1$  and $ n-2 $ are also  eigenvalues of $\Gamma$.
\end{proof}

We now try to find an orbit partition $\pi$ such that this partition has a singleton cell. If we construct such a partition for the vertex set of $ L(n) $, then by Theorem 2.5.  and Theorem 2.6.  every eigenvalue of the graph $L(n)$ is an eigenvalue of the matrix $P$, the adjacency matrix of the directed graph $ L(n) / \pi $.

Let $ K = \{  f_{ \alpha } \  |  \  \alpha \in Sym([n]), \alpha(1) = 1,  \alpha(2) = 2 \} $.  Then $K$ is a subgroup of $ Aut( \Gamma)$, the automorphism group of $\Gamma $.  Let $ H_2 = \{ \alpha \ | \ f_{ \alpha \  \in H} \}   $.  Note that $H_2$ is a subgroup of $Sym([n])$ isomorphic with $Sym([n-2])$. In the sequel,  we want to determined the orbit partition of  the subgroup $ K $. In fact,  $K$ generates the following orbits; \

\

$O_1 =K([1,12]) = \{   k([1,12]) \
| \  k \in \
K \}$ = $ \{ [ \alpha(1), \alpha(1)\alpha(2) ] \
\ | \ \alpha \in H_2\} $= $[1,12]$. \

$O_2 =K([1,13]) = \{   k([1,13]) \
| \  k \in \
K \}$ = $ \{ [ \alpha(1), \alpha(1)\alpha(3) ] \
\ | \   \alpha \in H_2\} $= $\{  [1,1i] \ | \ 3 \leq i \leq n \}$. \

$O_3 =K([2,12]) = \{   k([2,12]) \
| \  k \in \
K \}$ = $ \{ [ \alpha(2), \alpha(1)\alpha(2) ] \
\ | \ \alpha \in H_2\} $= $[2,12]$. \

$O_4 =K([2,23]) = \{   k([2,23]) \
| \  k \in \
K \}$ = $ \{ [ \alpha(2), \alpha(2)\alpha(3) ] \
\ |   \alpha \in H_2\} $= $\{  [2,2i] \ | \ 3 \leq i \leq n \}$.\

$O_5 =K([3,13]) = \{   k([3,13]) \
| \  k \in \
K \}$ = $ \{ [ \alpha(3), \alpha(1)\alpha(3) ] \
\ | \ \alpha \in H_2\} $= $\{  [i,1i] \ | \ 3 \leq i \leq n \}$. \

$O_6 =K([3,23]) = \{   k([3,23]) \
| \  k \in \
K \}$ = $ \{ [ \alpha(3), \alpha(2)\alpha(3) ] \
\ | \ \alpha \in H_2\} $= $\{  [i,2i] \ | \ 3 \leq i \leq n \}$. \

$O_7 =K([3,34]) = \{   k([3,34]) \
| \  k \in \
K \}$ = $ \{ [ \alpha(3), \alpha(3)\alpha(4) ] \
\ | \ \alpha \in H_2\} $= $\{  [i,ij] \ | \ 3 \leq i,j \leq n, \ i \neq j \}$. \

\

If we let $ \pi = \{ O_1 , O_2, O_3,...,O_7 \} $ and $p_{ij} $ is the number of arcs from  $O_i$ to $O_j$, then  $ O_1 \cup O_2 \cup ... \cup O_7 = V= V(\Gamma) $, and  the following hold. \
\

\

$p_{12}= n-2, \ p_{13}=1, \ and\  p_{1j}=0, j\neq 2,3 $, because the vertex $[1,12] \in O_1$ is adjacent to all the $n-2$ vertices in $O_2$, and $1$ vertex in  $O_3$, namely,  $[2,12]$  and therefore $0$ vertex in  other orbits.\
\

\

 $p_{21}= 1, \ p_{22}=n-3, \  \ p_{25}=1  \ and \   p_{2,j}=0, j\neq 1,2,5 $, because the vertex $[1,13] \in O_2$ is adjacent to $1$ vertex in  $O_1$, say, $ [1,12]$,  and  $n-3$ other vertices in  $O_2$, and $1$ vertex in  $O_5$,  namely,  $[3,31]$,   and therefore $0$ vertex in  other orbits.\
\

\

$p_{31}=1, \  p_{34}=n-2,   \  and \  p_{3,j}=0, j\neq 1,4 $, because the vertex $[2,21] \in O_3$ is adjacent to $1$ vertex in  $O_1$, namely,  $[1,12]$,  and all of the $n-2$ vertices in  $O_4$, and therefore $0$ vertex in  other orbits.\

\

$p_{43}= 1, \  p_{44}=n-3, \  p_{46}=1  \ and \   p_{4j}=0, j\neq 3,4,6 $, because the vertex $[2,23] \in O_4$ is adjacent to $1$ vertex in  $O_3$, say, $ [2,21]$, and  $n-3$ other vertices in  $O_4$ and $1$ vertex in  $O_6$, namely,  $[3,32]$, and therefore $0$ vertex in  other orbits.\

\

$p_{52}= 1, \ p_{56}=1,  \  \ p_{57}=n-3 \  and \  p_{5j}=0, j\neq 2,6,7 $, because the vertex $[3,31] \in O_5$ is adjacent to $1$ vertex in  $O_2$, namely,  $[1,13]$,   and  $1$ vertex in  $O_6$, say,  $[3,32]$,  and $ n-3$ vertices in  $O_7$, namely,  every vertex in  $O_7$ of the form $[3,3j], 4 \leq j \leq n$,  and therefore $0$ vertex in  other orbits.\

\

$p_{64}= 1, \ p_{65}=1, \  p_{67}=n-3 \ and \   p_{6j}=0, j\neq 4,5,7  $, because the vertex $[3,23] \in O_6$ is adjacent to $1$ vertex in $O_4$, say,  $[2,23]$, and   $1$ vertex in  $O_5$, namely,  $[3,31]$, and    $ n-3$ vertices in  $O_7$, namely,  every vertex in  $O_7$ of the form $[3,3j], 4 \leq j \leq n$,  and therefore $0$ vertex in  other orbits.\

\

$p_{75}= 1, \ p_{76}=1, \ p_{77}=n-3,  \ and \   p_{7j}=0, j\neq 5,6,7  $, because the vertex $[3,34] \in O_7$ is adjacent to $1$ vertex in  $O_5$, say,  $[3,31]$,  and $1$ vertex in  $O_6$, namely $[3,32]$,   and  $ n-3$ vertices in  $O_7$, namely,  every vertex in  $O_7$ of the form $[3,3j], 4 \leq j \leq n$ and therefore $0$ vertex in  other orbits.\

\

Therefore,  the following matrix $ P$ is an adjacency matrix for the directed graph $ \Gamma/ \pi$. \\

\centerline{$P$=$\begin{bmatrix}
0 & n-2 & 1 & 0 & 0 & 0 & 0 \\
 1 & n-3 & 0 & 0 & 1 & 0 & 0 \\
 1 & 0 & 0 & n-2 & 0 & 0 & 0 \\
 0 & 0 & 1 & n-3 & 0 & 1 & 0 \\
 0 & 1 & 0 & 0 & 0 & 1 & n-3 \\
 0 & 0 & 0 & 1 & 1 & 0 & n-3 \\
 0 & 0 & 0 & 0 & 1 & 1 & n-3
\end{bmatrix}$}
\

\

We can use the Mathematica program for finding the eigenvalues of the matrix $P$. By the Mathematica program, we have;  \
\

\

\noindent\(\pmb{m=\{\{0,n-2,1,0,0,0,0\},\{1,n-3,0,0,1,0,0\},\{1,0,0,n-2,0,0,0\},}\\
\pmb{\{0,0,1,n-3,0,1,0\},\{0,1,0,0,0,1,n-3\},\{0,0,0,1,1,0,n-3\},}\\
\pmb{\{0,0,0,0,1,1,n-3\}\}}\\$

\noindent\(\pmb{\text{Eigenvalues}[m]}\)=$  \{ - 2, - 1, - 1, 0, - 2 + n, - 2 + n, - 1 + n \} $
\

\

Since by Theorem 2.4,  Theorem 2.5. and   Theorem 2.6.  the set of distinct  eigenvalues of the graph $L(n)$ is equal to the  set of distinct  eigenvalues of the  matrix $P$, hence $  \{ - 2, - 1, 0,n-2,n-1 \} $ is the set of all  distinct eigenvalues of the graph $ L(n) $. We now have the following result.
\begin{theorem}
Let $n>3$ be an integer. Then the graph $L(n)$ is vertex-transitive  integral graph with distinct eigenvalues $ -2, -1, 0, n-2, n-1$.
\end{theorem} \

 \


\begin{thebibliography}{}


\bibitem{} N. L. Biggs, Algebraic Graph Theory (Second edition), Cambridge Mathematical Library
(Cambridge University Press, Cambridge, 1993).

\bibitem{}D. Cvetkovic, Z. Rodosavljevic,  and S.K. Simic, Errata, A survey on integral graphs, Univ.
Beograde. Publ. Elektrotehn. Fak. Ser. Mat. 15 (2004), 112.


\bibitem{}F. C. Bussemaker, D. Cvetkovic, There are exactly 13 connected, cubic, integral
graphs., Univ. Beograd, Publ. Elektrotehn. Fak., Ser. Mat., Fiz., Nos. 544-576
(1976), 43�48.


\bibitem{} A. E. Brouwer and W. H. Haemers, Spectra of Graphs, Springer, 2012.

\bibitem{} C. Godsil, G. Royle, Algebraic Graph Theory, Springer, 2001.


\bibitem{} F. Harary and A.J. Schwenk, Which graphs have integral spectra?, In Graphs and Combinatorics, (eds. R. Bari and F. Harary), (Proc. Capital Conf., George Washington Univ., Washington, D.C., 1973), Lecture Notes in Mathematics
406, Springer-Verlag, Berlin (1974), 45-51.

\end{thebibliography}
\end{document}